\documentclass[fleqn,12pt,twoside]{article}

\usepackage[headings]{espcrc1}
\readRCS
$Id: espcrc1.tex,v 1.2 2004/02/24 11:22:11 spepping Exp $
\usepackage{graphicx}
\usepackage[figuresright]{rotating}

\newtheorem{theorem}{Theorem}

\newtheorem{conjecture}[theorem]{Conjecture}
\newtheorem{corollary}[theorem]{Corollary}

\newenvironment{proof}[1][Proof.]{\begin{trivlist}
\item[\hskip \labelsep {\bfseries #1}]}{\end{trivlist}}

\newenvironment{acknowledgement}[1][Acknowledgement]{\begin{trivlist}
\item[\hskip \labelsep {\bfseries #1}]}{\end{trivlist}}

\newcommand{\AmS}{{\protect\the\textfont2
  A\kern-.1667em\lower.5ex\hbox{M}\kern-.125emS}}

\usepackage{amsmath}
\usepackage{amsfonts}
\usepackage{amssymb}
\title{On maximum matchings in almost regular graphs}

\author{Petros A. Petrosyan\address[MCSD]{Institute for Informatics and Automation Problems,\\
National Academy of Sciences, 0014, Armenia}%
\address{Department of Informatics and Applied Mathematics,\\
Yerevan State University, 0025, Armenia}%
\thanks{email: pet\_petros@\{ipia.sci.am, ysu.am, yahoo.com\}}}

\runtitle{On maximum matchings in almost regular
graphs}\runauthor{Petros A. Petrosyan}

\begin{document}

\maketitle

\begin{abstract}

In 2010, Mkrtchyan, Petrosyan and Vardanyan proved that every graph
$G$ with $2\leq \delta(G)\leq \Delta(G)\leq 3$ contains a maximum
matching whose unsaturated vertices do not have a common neighbor,
where $\Delta(G)$ and $\delta(G)$ denote the maximum and minimum
degrees of vertices in $G$, respectively. In the same paper they
suggested the following conjecture: every graph $G$ with
$\Delta(G)-\delta(G)\leq 1$ contains a maximum matching whose
unsaturated vertices do not have a common neighbor. Recently,
Picouleau disproved this conjecture by constructing a bipartite
counterexample $G$ with $\Delta(G)=5$ and $\delta(G)=4$. In this
note we show that the conjecture is false for graphs $G$ with
$\Delta(G)-\delta(G)=1$ and $\Delta(G)\geq 4$, and for $r$-regular graphs when $r\geq 7$.\\

\end{abstract}

\section{Introduction}\

Throughout this note all graphs are finite, undirected, and have no
loops, but may contain multiple edges. Let $V(G)$ and $E(G)$ denote
the sets of vertices and edges of $G$, respectively. For a graph
$G$, let $\Delta(G)$ and $\delta(G)$ denote the maximum and minimum
degrees of vertices in $G$, respectively. For two distinct vertices
$u$ and $v$ of a graph $G$, let $E(uv)$ denote the set of all edges
of $G$ joining $u$ with $v$. An $(a,b)$-biregular bipartite graph
$G$ is a bipartite graph $G$ with the vertices in one part all
having degree $a$ and the vertices in the other part all having
degree $b$. The terms and concepts that we do not define
can be found in \cite{b3}.\\

In \cite{b1}, Mkrtchyan, Petrosyan and Vardanyan proved the
following result.

\begin{theorem}
\label{mytheorem1} Every graph $G$ with $2\leq \delta(G)\leq
\Delta(G)\leq 3$ contains a maximum matching whose unsaturated
vertices do not have a common neighbor.
\end{theorem}

\begin{corollary}
\label{mycorollary1} Every cubic graph $G$ contains a maximum
matching whose unsaturated vertices do not have a common neighbor.
\end{corollary}

In the same paper they posed the following

\begin{conjecture}\label{vahanconjecture}
Every graph $G$ with $\Delta(G)-\delta(G)\leq 1$ contains a maximum
matching whose unsaturated vertices do not have a common neighbor.
\end{conjecture}

Also, they noted that they do not even know, whether this conjecture
holds for $r$-regular graphs with $r\geq 4$. In \cite{b2}, Picouleau
showed that Conjecture \ref{vahanconjecture} is false when $G$ is a
$(5,4)$-biregular bipartite graph. However, the question is still
open when $\Delta(G)=4$ and $\delta(G)=3$, and when
$\Delta(G)-\delta(G)=1$ with $\Delta(G)\geq 6$. Also, the question
remains open for $r$-regular graphs when $r\geq 4$.

In this note we prove that for any $r\geq 2$, there exists a
$(2r,2r-1)$-biregular bipartite graph $G$ such that for any maximum
matching of $G$ and any pair of unsaturated vertices with respect to
this maximum matching, these vertices have a common neighbor. Next
we show that for any $r\geq 3$, there exists a graph $G$ with
$\Delta(G)=2r+1$ and $\delta(G)=2r$ such that for any maximum
matching of $G$, there is a pair of unsaturated vertices
 with a common neighbor. Finally, we prove that for any $r\geq 3$,
 there exists a $(2r+1)$-regular graph $G$ such that for any
maximum matching of $G$, there is a pair of unsaturated vertices
with a common neighbor. We also construct the $8$-regular graph with
the same property and prove that for any $r\geq 5$, there exists a
$2r$-regular graph $G$ such that for any maximum matching of $G$ and
any pair of unsaturated vertices with respect to this maximum
matching, these vertices have a common neighbor.
\bigskip

\section{Results}\

First we consider graphs $G$ with $\Delta(G)-\delta(G)=1$.

\begin{theorem}
\label{mytheorem2} For any $r\geq 2$, there exists a
$(2r,2r-1)$-biregular bipartite graph $G$ such that for any maximum
matching of $G$ and any pair of unsaturated vertices with respect to
this maximum matching, these vertices have a common neighbor.
\end{theorem}
\begin{proof}
For the proof, we construct a graph $B_{2r,2r-1}$ for $r\geq 2$ that
satisfies the specified conditions. We define a graph $B_{2r,2r-1}$
as follows: $V\left(B_{2r,2r-1}\right)=U\cup V$, where
\begin{center}
$U=\left\{u_{(i,j)}\colon\,1\leq i<j\leq 2r\right\}$,
$V=\left\{v_{1}^{(i)},\ldots,v_{r}^{(i)}\colon\,1\leq i\leq
2r\right\}$, and
\end{center}

\begin{center}
$E\left(B_{2r,2r-1}\right)=\left\{v_{k}^{(i)}u_{(i,j)},v_{k}^{(j)}u_{(i,j)}\colon\,1\leq
i<j\leq 2r,1\leq k\leq r\right\}$.
\end{center}

Clearly, $B_{2r,2r-1}$ is a $(2r,2r-1)$-biregular bipartite graph
with a bipartition $(U,V)$. Moreover, let us note that $\vert U\vert
=\binom{2r}{2}=2r^{2}-r$ and $\vert V\vert =2r^{2}$. Thus,
$B_{2r,2r-1}$ has no perfect matching. On the other hand, by Hall's
theorem, it is not hard to see that each maximum matching $M$
saturates $U$. This implies that for any maximum matching $M$ of
$B_{2r,2r-1}$, we have $r$ unsaturated vertices from $V$. Now let
$v_{k_{0}}^{(i)}$ and $v_{k_{1}}^{(j)}$ be any pair of unsaturated
vertices from $V$ with respect to some maximum matching $M$. We
consider two cases.

Case 1: $i=j$.

In this case, by the construction of $B_{2r,2r-1}$,
$v_{k_{0}}^{(i)}$ and $v_{k_{1}}^{(i)}$ have a common neighbor
$u_{(i,l)}$ with $i<l$ or $u_{(l,i)}$ with $l<i$.

Case 2: $i\neq j$.

In this case, by the construction of $B_{2r,2r-1}$,
$v_{k_{0}}^{(i)}$ and $v_{k_{1}}^{(j)}$ have a common neighbor
$u_{(i,j)}$ if $i<j$ or $u_{(j,i)}$ if $j<i$.~$\square$
\end{proof}

\begin{theorem}
\label{mytheorem3} For any $r\geq 3$,
\begin{description}
\item[(1)] there exists a $(2r+1)$-regular graph $G$ such that for any
maximum matching of $G$, there is a pair of unsaturated vertices
with a common neighbor,
\item[(2)] there exists a graph $H$ with $\Delta(H)=2r+1$ and $\delta(H)=2r$
such that for any maximum matching of $H$, there is a pair of
unsaturated vertices with a common neighbor.
\end{description}
\end{theorem}
\begin{proof}(1)
For the proof, we are going to construct a graph $G_{2r+1}$ for
$r\geq 3$ that satisfies the specified conditions. We define a graph
$G_{2r+1}$ as follows:
\begin{center}
$V\left(G_{2r+1}\right)=\left\{x,y,z\right\}\cup
\left\{v_{1}^{(i)},v_{2}^{(i)},v_{3}^{(i)}\colon\,1\leq i\leq
2r+1\right\}$ and
\end{center}
\begin{center}
$E\left(G_{2r+1}\right)=\left\{xv_{1}^{(i)},yv_{2}^{(i)},zv_{3}^{(i)}\colon\,1\leq
i\leq 2r+1\right\}\cup
\left\{v_{1}^{(i)}v_{2}^{(i)},v_{2}^{(i)}v_{3}^{(i)},v_{3}^{(i)}v_{1}^{(i)}\colon\,
\left\vert E\left(v_{1}^{(i)}v_{2}^{(i)}\right)\right\vert
=\left\vert E\left(v_{2}^{(i)}v_{3}^{(i)}\right)\right\vert =
\left\vert E\left(v_{3}^{(i)}v_{1}^{(i)}\right)\right\vert = r,
1\leq i\leq 2r+1\right\}$.
\end{center}
Clearly, $G_{2r+1}$ is a $(2r+1)$-regular graph with $\left\vert
V\left(G_{2r+1}\right)\right\vert=6r+6$. By Tutte's theorem,
$G_{2r+1}$ has no perfect matching. On the other hand, it is not
hard to see that each maximum matching $M$ saturates $x,y$ and $z$.
This implies that for any maximum matching $M$ of $G_{2r+1}$, we
have $2r-2$ unsaturated vertices from
$V\left(G_{2r+1}\right)\setminus\{x,y,z\}$. Since $r\geq 3$, we have
that for any maximum matching $M$ of $G_{2r+1}$, the graph
$G_{2r+1}$ has at least four unsaturated vertices from
$V\left(G_{2r+1}\right)\setminus\{x,y,z\}$. However, by the
construction of $G_{2r+1}$, the vertices from
$V\left(G_{2r+1}\right)\setminus\{x,y,z\}$ have only three possible
values of the subindex; thus there are two vertices with the same
subindex. Let $v_{k}^{(i)}$ and $v_{k}^{(j)}$ be these unsaturated
vertices from $V\left(G_{2r+1}\right)\setminus\{x,y,z\}$ with
respect to some maximum matching $M$. If $k=1$, then, by the
construction of $G_{2r+1}$, $v_{1}^{(i)}$ and $v_{1}^{(j)}$ have a
common neighbor $x$, if $k=2$, then, by the construction of
$G_{2r+1}$, $v_{2}^{(i)}$ and $v_{2}^{(j)}$ have a common neighbor
$y$, and if $k=3$, then, by the construction of $G_{2r+1}$,
$v_{3}^{(i)}$ and $v_{3}^{(j)}$ have a common neighbor $z$.

(2) For the proof, it suffices to define a graph $H_{2r+1,2r}$ for
$r\geq 3$ as follows:
$V\left(H_{2r+1,2r}\right)=V\left(G_{2r+1}\right)$ and
$E\left(H_{2r+1,2r}\right)=E\left(G_{2r+1}\right)\setminus
\left\{v_{3}^{(i)}v_{1}^{(i)}\colon\,1\leq i\leq 2r+1\right\}$.
Clearly, $H_{2r+1,2r}$ is a graph with
$\Delta\left(H_{2r+1,2r}\right)=2r+1$ and
$\delta\left(H_{2r+1,2r}\right)=2r$. Similarly as in the proof of
(1) it can be shown that for any maximum matching of $H_{2r+1,2r}$,
there are two unsaturated vertices with a common neighbor.~$\square$
\end{proof}

These results combined with the result of Picouleau show that
Conjecture \ref{vahanconjecture} is false for graphs $G$ with
$\Delta(G)-\delta(G)=1$ and $\Delta(G)\geq 4$. Also, Conjecture
\ref{vahanconjecture} is false for $(2r+1)$-regular graphs when
$r\geq 3$. Next we consider even regular graphs. First we consider
the $8$-regular graph $G$ shown in Fig. \ref{8-regular graph}.
Similarly as in the proof of Theorem \ref{mytheorem3} it can be
shown that for any maximum matching of $G$, there are two
unsaturated vertices with a common neighbor.

\begin{figure}[h]
\begin{center}
\includegraphics[width=30pc,height=25pc]{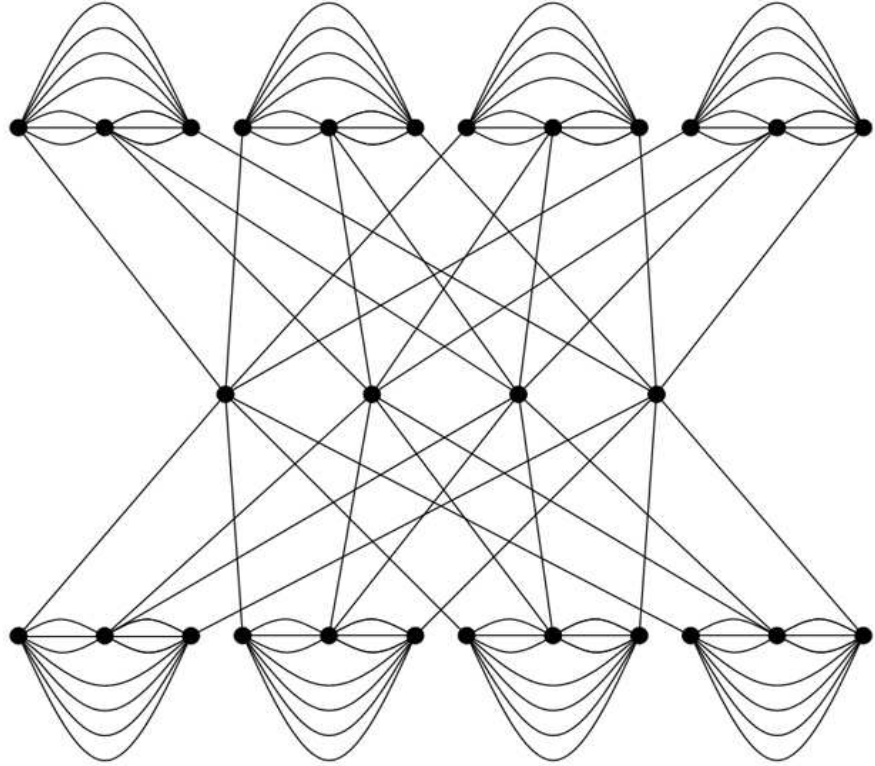}\\
\caption{The $8$-regular graph $G$.}\label{8-regular graph}
\end{center}
\end{figure}

\begin{theorem}
\label{mytheorem4} For any $r\geq 5$, there exists a $2r$-regular
graph $G$ such that for any maximum matching of $G$ and any pair of
unsaturated vertices with respect to this maximum matching, these
vertices have a common neighbor.
\end{theorem}
\begin{proof} For the proof, we construct a graph $F_{2r}$ for
$r\geq 5$ that satisfies the specified conditions. We define a graph
$F_{2r}$ as follows:
\begin{center}
$V\left(F_{2r}\right)=\left\{x,y,z\right\}\cup
\left\{v_{1}^{(i)},v_{2}^{(i)},v_{3}^{(i)}\colon\,1\leq i\leq
r\right\}$ and
\end{center}
\begin{center}
$E\left(F_{2r}\right)=\left\{xv_{1}^{(i)},xv_{2}^{(i)},yv_{1}^{(i)},yv_{3}^{(i)},zv_{2}^{(i)},zv_{3}^{(i)}\colon\,1\leq
i\leq r\right\}\cup
\left\{v_{1}^{(i)}v_{2}^{(i)},v_{2}^{(i)}v_{3}^{(i)},v_{3}^{(i)}v_{1}^{(i)}\colon\,
\left\vert E\left(v_{1}^{(i)}v_{2}^{(i)}\right)\right\vert
=\left\vert E\left(v_{2}^{(i)}v_{3}^{(i)}\right)\right\vert =
\left\vert E\left(v_{3}^{(i)}v_{1}^{(i)}\right)\right\vert = r-1,
1\leq i\leq r\right\}$.
\end{center}
Clearly, $F_{2r}$ is a $2r$-regular graph with $\left\vert
V\left(F_{2r}\right)\right\vert=3r+3$. By Tutte's theorem, $F_{2r}$
has no perfect matching. On the other hand, it is not hard to see
that each maximum matching $M$ saturates $x,y$ and $z$. This implies
that for any maximum matching $M$ of $F_{2r}$, we have $r-3$
unsaturated vertices from $V\left(F_{2r}\right)\setminus\{x,y,z\}$.
Since $r\geq 5$, we have that for any maximum matching $M$ of
$F_{2r}$, the graph $F_{2r}$ has at least two unsaturated vertices
from $V\left(F_{2r}\right)\setminus\{x,y,z\}$. Now let
$v_{k_{0}}^{(i)}$ and $v_{k_{1}}^{(j)}$ be any pair of unsaturated
vertices from $V\left(F_{2r}\right)\setminus\{x,y,z\}$ with respect
to some maximum matching $M$. We consider two cases.

Case 1: $k_{0}=k_{1}$.

If $k_{0}=k_{1}=1$ or $k_{0}=k_{1}=2$, then, by the construction of
$F_{2r}$, $v_{k_{0}}^{(i)}$ and $v_{k_{1}}^{(j)}$ have a common
neighbor $x$. If $k_{0}=k_{1}=3$, then, by the construction of
$F_{2r}$, $v_{k_{0}}^{(i)}$ and $v_{k_{1}}^{(j)}$ have a common
neighbor $y$.

Case 2: $k_{0}\neq k_{1}$.

If $(k_{0},k_{1})=(1,2)$, then, by the construction of $F_{2r}$,
$v_{k_{0}}^{(i)}$ and $v_{k_{1}}^{(j)}$ have a common neighbor $x$,
if $(k_{0},k_{1})=(1,3)$, then, by the construction of $F_{2r}$,
$v_{k_{0}}^{(i)}$ and $v_{k_{1}}^{(j)}$ have a common neighbor $y$,
and if $(k_{0},k_{1})=(2,3)$, then, by the construction of $F_{2r}$,
$v_{k_{0}}^{(i)}$ and $v_{k_{1}}^{(j)}$ have a common neighbor
$z$.~$\square$
\end{proof}

Our results show that Conjecture \ref{vahanconjecture} is also false
for $r$-regular graphs when $r\geq 7$. Thus, the question remains
open only for $4,5$ and $6$-regular graphs.

\begin{acknowledgement}
We would like to thank the anonymous referees for useful
suggestions.
\end{acknowledgement}


\begin{thebibliography}{99}

\bibitem{b1} V.V. Mkrtchyan, S.S. Petrosyan, G.N. Vardanyan, On disjoint matchings in cubic graphs,
Discrete Math. 310 (2010) 1588-1613.

\bibitem{b2} C. Picouleau, A note on a conjecture on maximum matching in almost regular graphs,
Discrete Math. 310 (2010) 3646-3647.

\bibitem{b3} D.B. West, Introduction to Graph Theory, Prentice-Hall, New
Jersey, 2001.

\end{thebibliography}
\end{document}